\RequirePackage{fix-cm}
\documentclass[smallextended]{svjour3}       
\smartqed  
\usepackage{amsmath,amsfonts,mathrsfs}
\usepackage{array,color}
\usepackage{tikz}
\usepackage[caption=false]{subfig}
\usepackage{float}
\usepackage{bm,amssymb}
\newcommand{\be}{\begin{eqnarray}}
\newcommand{\ee}{\end{eqnarray}}
\newcommand{\tv}{\bm{v}}

\usepackage{graphicx}
\begin{document}

\title{Nonconforming finite element methods of order two and order three for the Stokes flow in three dimensions\thanks{The first two authors were supported by the National Natural Science Foundation of China grants NSFC 12288101.}
}

\titlerunning{Nonconforming elements for the Stokes flow in 3D}        

\author{Wei Chen         \and
           Jun Hu              \and
           Min Zhang
}

\authorrunning{W. Chen, J. Hu, M. Zhang} 

\institute{Wei Chen  \at
              LMAM and School of Mathematical Sciences, Peking University, Beijing 100871, P. R. China \\
              \email{1901110037@pku.edu.cn}          
           \and
           Jun Hu \at
           LMAM and School of Mathematical Sciences, Peking University, Beijing 100871, P. R. China \\
           \email{hujun@math.pku.edu.cn}
            \and
           Min Zhang \at
            LMAM and School of Mathematical Sciences, Peking University, Beijing 100871, P. R. China \\
             \emph{Present address:Computational Science Research Center, Beijing 100193, P. R. China}  \\
            \email{zhangminyy@csrc.ac.cn}
}

\date{Received: date / Accepted: date}

\maketitle

\begin{abstract}
In this study, the nonconforming finite elements of order two and order three are constructed and exploited for the Stokes problem. The moments of order up to $k-1$ ($k=2,3$) on all the facets of the tetrahedron are used for DoFs (degrees of freedom) to construct the unisolvent $k$-order nonconforming finite element with the bubble function space of $P_{k+1}$ explicitly represented. The pair of the $k$-order element and the discontinuous piecewise $P_{k}$ is proved to be stable for solving the Stokes problem with the element-wise divergence-free condition preserved. The main difficulty in establishing the discrete inf-sup condition comes from the fact that the usual Fortin operator can not be constructed. Thanks to the explicit representation of the bubble functions, its divergence space is proved to be identical to the orthogonal complement space of constants with respect to $P_k$ on the tetrahedron, which plays an important role to overcome the aforementioned difficulty and leads to the desirable well-posedness of the discrete problem. Furthermore, a reduced $k$-order nonconforming finite element with a discontinuous piecewise $P_{k-1}$ is designed and proved to be stable for solving the Stokes problem. The lack of the Fortin operator causes difficulty in analyzing the discrete inf-sup condition for the reduced third-order element pair. To deal with this problem, the so-called macro-element technique is adopted with a crucial algebraic result concerning the property of functions in the orthogonal complement space of the divergence of the discrete velocity space with respect to the discrete pressure space on the macro-element. Numerical experiments are provided to validate the theoretical results.
\end{abstract}

\section{Introduction}
\label{intro}
Assume that $\Omega\subseteq \mathbb{R}^3$ is a bounded, polyhedral domain with Lipschitz boundary $\partial\Omega$. Let $\Gamma_D$ and $\Gamma_N = \partial\Omega \backslash \Gamma_D$ denote the closed Dirichlet boundary and Neumann boundary of $\Omega$ respectively, both with nonzero two-dimensional measures. For $\bm{f} \in \bm{L}^2(\Omega)$, $\bm{u}_D \in \bm{H}^{3/2}(\Gamma_D)$, and $\bm{g} \in \bm{H}^{1/2}(\Gamma_N)$, the three-dimensional Stokes problem with homogeneous Dirichlet boundary conditions seeks $(\bm{u}, p) \in \bm{H}_D^1(\Omega) \times L^2(\Omega),$ such that
\begin{equation}
\label{eq1}
\left\{
\begin{aligned}
	a(\bm{u},\bm{v})+b(\bm{v},p)&=(\bm{f},\bm{v})+\int_{\Gamma_N}\bm{g}\cdot\bm{v}ds &\text{ for all }\bm{v}\in\bm{H}_D^1(\Omega),\\
	b(\bm{u},q)&=0&\text{ for all } q\in L^2(\Omega),
\end{aligned}
\right.
\end{equation}
where
\begin{align*}
	a(\bm{u},\bm{v}):=2\mu\int_\Omega\varepsilon(\bm{u}):\varepsilon(\bm{v})\, d{x},\quad b(\bm{v},q):=-\int_\Omega \operatorname{div}\bm{v}\, q\, d{x},
\end{align*}
and $\bm{H}^1_D(\Omega)$ is the space of all $\bm{H}^1(\Omega)$ functions that vanish on the Dirichlet boundary $\Gamma_D$ in the sense of traces. 
Here $\mu>0$ is the viscosity, $\varepsilon(\bm{u}):=(\nabla \bm{u}+\nabla \bm{u}^{\mathsf{T}})/2$ is the deformation rate tensor, and $\sigma:=(2\mu\varepsilon(\bm{u})-p\bm{I})$ is the symmetric stress tensor with $\bm{I}$ denoting the identity matrix.

A stable finite element pair for solving \eqref{eq1} should satisfy the discrete inf-sup (or Lady\v{z}enskaja-Babu\v{s}ka-Brezzi) condition and the discrete Korn's inequality \cite{brezzi2012mixed,boffi2013mixed}. Besides, the (element-wise) divergence-free velocity approximation is desirable. 
In general, the (element-wise) divergence space of the discrete velocity space is required to be equal to the discrete pressure space in order that both the discrete inf-sup condition and the (element-wise) divergence-free condition could hold.
However, this is usually hard to meet.

The construction of two-dimensional finite element pairs is rather matured, see
\cite{boffi2013mixed,brezzi2012mixed,Scott1985NormEF,Girault1986FiniteEM,Qin2007StabilityAA,Shangyou2008ONTP,Zhang2016StableFE,Chen2017ACS,1994Boffi,2019guz,MR339677,MR1098408,2013Guzma2D}. 
In contrast, only a few studies on the three-dimensional problem are available.
The conforming finite elements for solving \eqref{eq1} include the MINI element \cite{arnold1984stable} 
and the Taylor-Hood element \cite{1997Boffi}. 
Several stable conforming finite element pairs with divergence-free velocity fields are developed on special triangulations \cite{Zhang2005ANF,2016Neilan,Zhang2011QuadraticDF,2011Zhang,2018guz,2022guz}; or constructed by augmenting the velocity space with rational functions on tetrahedral grids \cite{2013Guzma3D}. As for nonconforming finite elements for the Stokes problem,
the Crouzeix-Raviart element \cite{m1973p} does not satisfy the discrete Korn's inequality. To circumvent this, some first-order stable nonconforming finite elements were developed on tetrahedral grids \cite{hu2018two} and on parallelepiped grids \cite{zhangm2017zhangsy} by exploring a discrete velocity space consisting of two conforming components and one nonconforming component.
The construction of nonconforming finite elements of order greater than one is challenging. In \cite{1985F}, a second-order nonconforming element was developed by enriching the quadratic Lagrange element with some extra nonconforming bubble functions. This element is stable and element-wise divergence-free, but the DoFs cannot be defined in the sense of Ciarlet \cite[Definition 3.1.1]{brenner2007mathematical}.
Besides, see \cite{tai2006discrete,zhang2022nonconforming,johnny2012family} for $\bm{H}(\operatorname{div})$ conforming finite elements.

The purpose of this paper is to propose nonconforming finite elements of order two and order three on tetrahedral grids and apply them to solve the Stokes problem \eqref{eq1}. In the spirit of the Crouzeix-Raviart element, the moments of order up to $k-1$ on all the facets of the tetrahedron are defined as DoFs to construct the $k$-order nonconforming finite element for $k=2,3$.
However, the number of these DoFs is greater than the dimension of $P_k(T)$ (polynomials of degree up to $k$ on the tetrahedron $T$). Moreover, there is a bubble function in $P_{k}(T)$ vanishing at all of the DoFs of moments on all the facets of $T$. Therefore, the shape function space has to be taken as $P_{k+1}(T)$. Hence, the bubble function space of $P_{k+1}(T)$ concerning the DoFs of moments on all the facets of $T$ must be characterized.

One main ingredient of this paper is to figure out the bubble function space of $P_{k+1}(T)$ by giving the explicit representation of bubble functions and proposing a direct sum decomposition of $P_{k+1}(T)$. By virtue of this, it is proved that the dimension of the bubble function space is 8 for $k=2$, and 11 for $k=3$. Thus, the dimension of the bubble function space is equal to the dimension of $P_{k+1}(T)$ minus the number of the DoFs of moments on all the facets of $T$. This implies the linear independence of the DoFs of moments on all the facets on $P_{k+1}(T)$. Based on these results, some interior DoFs in $T$ are designed to determine the bubble function space. These interior DoFs and the DoFs of moments on all the facets form a set of DoFs for the shape function space $P_{k+1}(T)$, which is unisolvent and results in a $k$-order nonconforming finite element. Besides, the nonconforming finite element of order $k$ is reduced by getting rid of some interior DoFs while keeping the DoFs of moments on all the facets, which leads to a reduced $k$-order nonconforming finite element. 

Next, the $k$-order nonconforming finite element with the discontinuous piecewise $P_{k}(T)$ is used to solve \eqref{eq1}. It should be noted that the element-wise divergence space of the discrete velocity space is a subspace of the discrete pressure space. Thus the discrete velocity field is element-wise divergence-free. However, it is difficult to establish the discrete inf-sup condition, since the usual Fortin operator \cite[Proposition 2.8]{brezzi2012mixed} can not be constructed for the $k$-order element pair. To overcome this difficulty, the key is to prove the divergence space of the bubble function space is identical to the orthogonal complement space of constants with respect to $P_k(T)$ on the tetrahedron. This and a new technique introduced in \cite{MR3301063} yield the discrete inf-sup condition. The Korn's inequality can be derived immediately by the continuity of the moments of order up to $k-1$ on each inter-element facet in light of the discussion in \cite{brenner2004korn}. Therefore, the well-posedness of the discrete problem is obtained, and the error estimates are proved by the standard mixed finite element theory \cite[Proposition 5.5.6]{boffi2013mixed} and the consistency error estimates \cite[Lemma 3]{m1973p}.

Furthermore, the reduced $k$-order nonconforming finite element with the discontinuous piecewise $P_{k-1}(T)$ is proved to be stable for solving \eqref{eq1}. It has fewer DoFs while retaining the same convergence rate. For the reduced third-order element pair, the discrete inf-sup condition is hard to establish due to the lack of the Fortin operator \cite[Proposition 2.8]{brezzi2012mixed}. To overcome this difficulty, the macro-element technique introduced in \cite{stenberg1984analysis} is adopted with a crucial algebraic result concerning the the property of functions in the orthogonal complement space of the divergence of the discrete velocity space with respect to the discrete pressure space on the macro-element. In fact, for the function therein, the values at the mid-points are proved to be equal on two opposite edges of any tetrahedron in the macro-element. This plus some specially defined bubble functions yields the discrete inf-sup condition. In addition, the discrete Korn's inequality is from \cite{brenner2004korn}. The well-posedness of the discrete problem and error estimates are obtained in the routine way.

The rest of the paper is organized as follows. Some preliminaries are introduced in Section 2. In Section 3, the bubble function space is characterized first, and based on this, the nonconforming and reduced nonconforming finite elements of order two and order three are constructed. The four newly proposed elements are exploited to solve the Stokes problem in Section 4. In Section 5, the discrete inf-sup condition is established. The well-posedness of the discrete problem and error estimates are presented in Section 6. Finally, the numerical experiments are performed in Section 7 to validate the previous results.

\section{Preliminaries}
Suppose $\Omega$ is subdivided by a family of quasi-uniform tetrahedral grids $\{\mathcal{T}_{h}\}$ with grid size $h$. Let $\mathcal{F}_h(\Omega)$, $\mathcal{V}_h(\Omega)$ and $\mathcal{F}_h(\Gamma_D)$, represent the set of interior faces of $\mathcal{T}_h$, the set of internal vertices of $\mathcal{T}_h$, and the set of faces on $\Gamma_D$, respectively.
Given any face $F\in\mathcal{F}_h(\Omega)$ with diameter $h_F$, a fixed outward unit normal $\mathbf{n}_F$ is assigned. Then there exist tetrahedrons $T_-$ and $T_+$ in $\mathcal{T}_h$ such that $F=T_+\cap T_-$. For a $\mathbb{R}^d$-valued function $\bm{v}$ with $d\in \{ 1,2,3\}$, define 
\begin{align*}
    [\![\bm{v}]\!]:=({\bm{v}}|_{T_{+}})|_{F}-({\bm{v}}|_{T_{-}})|_{F},\quad \{\!\!\{\bm{v}\}\!\!\}:=\frac{({\bm{v}}|_{T_{+}})|_{F}+({\bm{v}}|_{T_{-}})|_{F}}{2}. 
\end{align*}
The jump $[\![\bm{v}]\!]$ and the average $\{\!\!\{\bm{v}\}\!\!\}$ across $F$ reduce to the traces when $F$ lies on the boundary.

Given a tetrahedron $T\in \mathcal{T}_{h}$, the set of all faces and all vertices of $T$ are denoted by $\mathcal{F}(T)$ and $\mathcal{V}(T)$, respectively. Let $a_i\in\mathcal{V}(T)$ be the vertices of $T$, $i=1,2,3,4$. Let $F_i\in \mathcal{F}(T)$ be the face opposite to $a_i$. Let $\lambda_i$ denote the $i$-th barycentric coordinate of $T$. Let $\vert{F_i}\vert$ and $\vert{T}\vert$ be the area of the face $F_i$ and the volume of the tetrahedron $T$, respectively. 

Throughout this paper, an inequality $\alpha \lesssim \beta$ replaces $\alpha\leq c\beta$ with the multiplicative mesh-size independent constant $c>0$, which depends on $\Omega$ only.
Standard notations for Hilbert spaces $H^m(\Omega)$ and their associated norms $\|\cdot\|_m$ and seminorms $|\cdot|_m$ are employed. 
For a subdomain $S\subseteq \Omega$, $(.,.)_{S}$ denotes the $L^2$ scalar product over $S$, and $\Vert{\cdot}\Vert_{m,S}$ denotes the $H^m(S)$ norm over $S$. Let $(\cdot,\cdot)$ and $\Vert{\cdot}\Vert_m$ abbreviate $(\cdot,\cdot)_{\Omega}$ and $\Vert{\cdot}\Vert_{m,\Omega}$, respectively. For an integer $k\geq 0$, let $P_k(S)$ denote the set of all polynomials on $S$ with degree no more than $k$.
If there is no special instruction, the boldface letter will indicate a vector or vector space in order to distinguish it from scalars, for instance, $\bm{P}_k(S)=P_k(S)\times P_k(S)\times P_k(S)$.

For piecewise functions in nonconforming spaces, let $\operatorname{div}_h$ denote the element-wise divergence operator.

\subsection{{$M_{k-1}$}-continuity} 
The construction of the nonconforming finite element is based on the following assumptions: 

\noindent(H)\, The $M_{k-1}$-continuous condition: For any $\mathbb{R}^d$-valued piecewise polynomial function $\bm{v}$ with $d\in\{1,2, 3\}$, the moments
\begin{equation*}
\int_{F}[\![\bm{v}]\!]\cdot \bm{q}\, ds=0\quad \text{for all}~~\bm{q} \in \bm{P}_{k-1}(F)\, \, \text{and}\,\, F\in \mathcal{F}_{h}(\Omega).
\end{equation*}
\noindent(H\,$^{\prime}$)\, The $M_{k-1}$-zero condition: For any $\mathbb{R}^d$-valued piecewise polynomial function $\bm{v}$ with $d\in\{1,2, 3\}$, the moments 
\begin{equation*}
\int_{F}\bm{v}\cdot\bm{q}\, ds=0\quad \text{for all}~~\bm{q} \in \bm{P}_{k-1}(F)\, \, \text{and}\,\, F\in \mathcal{F}_{h}(\Gamma_{D}).
\end{equation*}
\section{The construction of the nonconforming finite elements} 
Given a tetrahedron $T\in\mathcal{T}_h$, the moments of order up to $k-1$ on each face $F\in\mathcal{F}(T)$ are defined as DoFs to ensure (H) for the $k$-order($k=2,3$) nonconforming finite element. However, the number of these DoFs is greater than the dimension of $P_k(T)$. In addition, there is a bubble function in $P_{k}(T)$ vanishing at the DoFs of moments on all the facets of $T$. Thus, the shape function space has to be taken as $P_{k+1}(T)$. Hence, the bubble function space of $P_{k+1}(T)$ is characterized first in this section. 
Based on this, some interior DoFs in $T$ are designed to determine the nontrivial bubble functions, which results in the $k$-order nonconforming finite element in the second part of this section. Besides, the element is reduced by getting rid of some interior DoFs on the tetrahedron, while keeping the DoFs of moments on all the facets, which leads to the reduced $k$-order nonconforming finite element in the last part of this section.

\subsection{The characterization of the bubble function space} The bubble function space of $P_{k+1}(T)$ concerning the DoFs of moments on all the facets of $T$ is characterized in this subsection. Specifically, its dimension is figured out, the explicit expression of the bubble functions is given, and a direct sum decomposition of $P_{k+1}(T)$ is proposed. This is very important for the construction in subsequent subsections.

First off, the bubble function space of $P_{k+1}(T)$ is defined by
\begin{align}
	\label{P3bubblespace}
	\mathfrak{B}_{k+1}(T):=\left\{v\in P_{k+1}(T)~|~ \int_F vq~ds=0\text{ for all }q\in P_{k-1}(F), F\in \mathcal{F}(T)\right\}.
\end{align}
Then the dimension of $\mathfrak{B}_{k+1}(T)$ is figured out in the following lemma. The proof is not easy, but the result plays a prominent part in the construction.

\begin{lemma}\label{dimB3--B4}
	There holds
	\begin{align}
&\operatorname{dim}\mathfrak{B}_3(T)=8, \quad \operatorname{dim}\mathfrak{B}_4(T)=11.\label{dimB4}
	\end{align}
\end{lemma}
\begin{proof}
	For $k=2$, define
	\begin{align}
		&\phi_i = 5\lambda_i^3-5\lambda_i^2+\lambda_i,\quad i = 1,2, 3, 4,\label{B3bb1}\\
		&\varphi_i = -2\lambda_i^2+4\lambda_i-2 + 3\sum_{j=1, j\neq i}^4\lambda_i^2+30\prod_{j=1,j\neq i}^4\lambda_j,\quad i = 1,2, 3, 4.\label{B3bb2}
	\end{align}
These eight functions are in $\mathfrak{B}_3(T)$ and linearly independent. Thus 
 \be\label{dim:2:geq}
 \operatorname{dim} \mathfrak{B}_{3}(T)\geq 8.
 \ee
	On the other hand, define 
	\be\label{orth:mathbbO:3}
	\mathfrak{O}_3(T):=\text{span}\{\lambda_i^2\lambda_j, \lambda_i\lambda_j^2 ~|~ 1\leq i<j \leq 4\}.
	\ee
	It can be verified that $\operatorname{dim}\mathfrak{O}_3(T)=12$ and $\mathfrak{O}_3(T)\subseteq P_3(T)$. Furthermore, $\mathfrak{O}_3(T)\cap\mathfrak{B}_3(T)=\{0\}$.  To prove this, given $v\in \mathfrak{O}_3(T)\cap \mathfrak{B}_3(T)$, then 
	\begin{equation}\label{2:v:form}
	v=\sum_{1\leq i< j\leq 4}\left(c_{i,j}\lambda_i^2\lambda_j+d_{i,j}\lambda_{i}\lambda_j^2\right) 
	\end{equation} 
	with parameters $c_{i,j}$ and $d_{i,j}$, and
 \be
	\label{12uni}
	\mathscr{F}_{i,j}(v):=\frac{1}{\vert{F_i}\vert}\int_{F_i} v \lambda_j d{s}=0, \quad i, j\in \{1, 2, 3, 4\},\,  i\neq j.
	\ee
	Let $\mathscr{F}$ be the linear functional of the following form:
	\be\label{functional:2nd:1}
	\begin{aligned}
	\mathscr{F}:&=\mathscr{F}_{1,3}-5\mathscr{F}_{1,2}+\mathscr{F}_{1,4}-2\mathscr{F}_{2,3}-2\mathscr{F}_{2,4}+10\mathscr{F}_{2,1}\\
		&\quad +7\mathscr{F}_{3,2}-11\mathscr{F}_{3,1} +\mathscr{F}_{3,4}+ 7\mathscr{F}_{4,2}-11\mathscr{F}_{4,1}+\mathscr{F}_{4,3}.
	\end{aligned}
	\ee
	Substituting \eqref{2:v:form} into \eqref{functional:2nd:1} brings
	$0=\mathscr{F}(v)=-\frac{1}{5}c_{1,2}$. Thus, $c_{1,2}=0$. Similar arguments lead to $c_{i,j}=d_{i,j}=0$ for all $1\leq i<j\leq 4$. This shows $v=0$. Therefore,
	\begin{align*}
		\dim \mathfrak{B}_3(T)&\leq \dim P_3(T)-\dim \mathfrak{O}_3(T)= 20 - 12=8.
	\end{align*}
	This plus \eqref{dim:2:geq} yields $\operatorname{dim}\mathfrak{B}_{3}(T)=8$.
	
	For $k=3$, proceeding as the proof of the case $k=2$, it can be verified that
	\begin{align}
		&\phi_i = 14\lambda_i^4-21\lambda_i^3+9\lambda_i^2-\lambda_i,~~i = 1, 2, 3, 4,\label{B4bb1}\\
		&\phi =\lambda_1\lambda_2\lambda_3\lambda_4,\label{B4bb3}
	\end{align}
	and
	\begin{align}
 \textcolor{black}{\varphi_{ij}} &= 28(\lambda_i+\lambda_j)^4-53(\lambda_i+\lambda_j)^3+27(\lambda_i+\lambda_j)^2-18\lambda_i\lambda_j + 3\lambda_k\lambda_l\notag\\ 
		&\qquad + (\lambda_i+\lambda_j)\left(21\lambda_i\lambda_j - 2 - 3(\lambda_k+\lambda_l)^2 - 21\lambda_k\lambda_l\right)\label{B4bb2}
	\end{align}
	with distinct integers $i,j,k,l$ in $\{1,2,3,4\}$ satisfying $k<l$ and $i<j$, are linearly independent, and these eleven bubble functions belong to $\mathfrak{B}_4(T)$. Therefore,
	\be
	\label{dimB4geq11}
	\dim \mathfrak{B}_4(T)\geq 11.
	\ee
Furthermore, define
	\be\label{orth:mathbbO:4}
	\begin{aligned}
		\mathfrak{O}_4(T):=\text{span}\{\lambda_i^3\lambda_j,&\lambda_i\lambda_j^3~|~1\leq i<j\leq 4\}\oplus\\
		&\text{span}\{\lambda_i^2\lambda_j\lambda_k,\lambda_i\lambda_j^2\lambda_k,\lambda_i\lambda_j\lambda_k^2~|~ 1\leq i<j<k\leq 4\}.
	\end{aligned}
	\ee
	Note that $\mathfrak{O}_4(T)\subseteq P_4(T)$ and $\operatorname{dim}\mathfrak{O}_4(T)=24$. The similar technique as used in the case $k=2$ shows $\mathfrak{B}_4(T)\cap\mathfrak{O}_4(T)=\{0\}$. This and \eqref{dimB4geq11} result in $\operatorname{dim}\mathfrak{B}_4(T)= 11$.
\end{proof}
\begin{remark}
The explicit representation of eight linearly independent bubble functions in $\mathfrak{B}_3(T)$ is given by \eqref{B3bb1}--\eqref{B3bb2}. Eleven linearly independent bubble functions in $\mathfrak{B}_4(T)$ are presented in \eqref{B4bb1}--\eqref{B4bb2}.  
\end{remark}
In fact, the proof of Lemma \ref{dimB3--B4} implies the following direct sum composition of $P_{k+1}(T)$:
\begin{align}
	P_{k+1}(T)=\mathfrak{B}_{k+1}(T)\oplus \mathfrak{O}_{k+1}(T)\quad \text{for}\, \, k=2, 3.\label{directsum}
\end{align}
Furthermore, note that the number of the DoFs of moments on all the facets of $T$ is $2k(k+1)$, which is equal to $\operatorname{dim}\mathfrak{O}_{k+1}$ for $k=2,3$. Therefore, the following lemma can be derived immediately.
\begin{lemma}\label{ifandonlyiflm}
For $k=2, 3$, it holds that
$\operatorname{dim} P_{k+1}(T) = \operatorname{dim}\mathfrak{B}_{k+1}(T) + 2k(k+1)$.
\end{lemma}

\subsection{The nonconforming finite elements of order two and order three}
The characterization of the bubble function space enables us to construct the $k$-order nonconforming finite element in this subsection for $k=2, 3$. 
On the tetrahedron $T\in \mathcal{T}_h$, let the shape function space be $P_{k+1}(T)$. For $v\in P_{k+1}(T)$, the DoFs are given by 
\begin{align}   
	& \int_F vq~ds, &&\text{ for all } q\in P_{k-1}(F), F\in\mathcal{F}(T),\label{facemomentP3}\\ 
	&\int_T vq~dx, &&\text{ for all } q\in \mathfrak{B}_{k+1}(T). \label{momentP3}
\end{align}
\begin{theorem}
Any function $v\in P_{k+1}(T)$ can be uniquely determined by DoFs \eqref{facemomentP3}--\eqref{momentP3}.
\end{theorem}
\begin{proof}
For $k=2, 3$, the dimension of $P_{k+1}(T)$ is equally the the number of the DoFs \eqref{facemomentP3}--\eqref{momentP3} in light of  Lemma \ref{ifandonlyiflm}.
It suffices to prove $v=0$ provided that \eqref{facemomentP3}--\eqref{momentP3} vanish. The vanishing of \eqref{facemomentP3} shows $v\in \mathfrak{B}_{k+1}(T)$, then \eqref{momentP3} concludes the proof.
\end{proof}
The global space for the $k$-order nonconforming finite element is defined by
\begin{align}
V_{h,k+1}:= \{v\in L^2(\Omega)|\,   v|_T\in P_{k+1}(T) \text{ for all } T\in\mathcal{T}_h, v~\text{is $M_{k-1}$-continuous}\}.\label{H2n}
\end{align}

\subsection{The reduced nonconforming finite elements of order two and order three}
For $k=2, 3$, the aforementioned $k$-order nonconforming finite element is reduced by getting rid of some interior DoFs on the tetrahedron, while keeping the DoFs of moments on all the facets for the $M_{k-1}$-continuity. 
Equivalently, a reduced shape function space, which is denoted by $P_{k+1}^-(T)$, is obtained by by enriching $P_{k}(T)$ with some $k+1$-order functions.
The main issue is how many and which $k+1$-order functions to choose to enrich $P_{k}(T)$.

To begin with, the explicit expression of the bubble function in $P_{k}(T)$  is presented in the following lemma. The proof is simple, while the result is important for determining the number of supplemental polynomials enriched on $P_{k}(T)$.

\begin{lemma}\label{mark3.4}
For $k=2, 3$, there exists the bubble function
\begin{align}
    &\mathfrak{b}_{2}:=2-4\sum_{i=1}^4\lambda_i^2~~\text{in}~~P_2(T),\label{bubulep2}\\
    &\mathfrak{b}_{3}:=\sum_{i=1}^4\left(9\lambda_i^2-11\lambda_i^3\right)+72\sum_{1\leq i< j< k\leq 4}\lambda_i\lambda_j\lambda_k~~~\text{in}~~P_3(T),\label{bubulep3}
\end{align}
such that the moments of order up to $k-1$ of $\mathfrak{b}_k$ on all the facets of the tetrahedron $T\in\mathcal{T}_h$ vanish, and the dimension of the bubble function space of $P_k(T)$ is one. 
\end{lemma}

In light of Lemma \ref{mark3.4}, for $P_2(T)$, there is a bubble function $\mathfrak{b}_2$ vanishing at all of the 12 DoFs of moments on all the facets of $T$. This implies the dimension of $P_{3}^-(T)$ should be no less than 13. Therefore, to obtain $P_3^-(T)$, at least three linearly independent cubic functions should be enriched on $P_2(T)$. Similarly, for $P_3(T)$, there is a bubble function $\mathfrak{b}_3$ vanishing at all of the 24 DoFs on all the facets of $T$. Thus, the dimension of $P_{4}^-(T)$ should be no less than 25. Therefore, to obtain $P_4^-(T)$, at least five linearly independent quatic functions should be enriched on $P_3(T)$.

Then motivated by the direct sum decomposition \eqref{directsum}, the shape function spaces are defined by 
\begin{align}
    &P_3^-(T):=P_2(T)+\text{span}\left\{\lambda_1^2\lambda_2,~\lambda_2^2\lambda_3,~\lambda_3^2\lambda_1\right\},\label{shape:1}\\
    &P_4^-(T):=P_3(T)+\text{span}\left\{\lambda_1^2\lambda_3\lambda_4,~\lambda_2^2\lambda_4\lambda_1,~\lambda_3^2\lambda_1\lambda_2,~\lambda_4^2\lambda_2\lambda_3,~\lambda_1\lambda_2^3\right\}.\label{shape:2}
\end{align}

\begin{remark}
For $P_{k+1}^-(T)$ with $k=2, 3$, the $k+1$-order supplemental polynomials enriched on $P_{k}(T)$ are selected from $\mathfrak{O}_{k+1}(T)$ in \eqref{directsum} directly. It should be noted that the choice is not unique. In principle, if there is no bubble function except $\mathfrak{b}_k$ in $P^-_{k+1}(T)$, $P^-_{k+1}(T)$ is workable.
\end{remark}

For $k=2, 3$, on the tetrahedron $T\in \mathcal{T}_h$, given any function $v\in P_{k+1}^-(T)$, the DoFs of the reduced nonconforming finite element are given by 
\begin{align}
	& \int_F vq~ds, &&\text{ for all } q\in P_{k-1}(F), F\in\mathcal{F}(T),\label{P2elemdofa}\\ 
	&\int_T v~dx. &&~ \label{P2elemdofb}
\end{align}

\begin{theorem}
\label{P2elemThm}
Any function $v\in P_{k+1}^-(T)$ can be uniquely determined by DoFs \eqref{P2elemdofa}--\eqref{P2elemdofb}.
\end{theorem}
\begin{proof}
The number of the DoFs \eqref{P2elemdofa}--\eqref{P2elemdofb} is  $2k(k+1)+1$, which is equally the dimension of the shape function space $P_{k+1}^-(T)$. It suffices to prove $v=0$ provided that  \eqref{P2elemdofa}--\eqref{P2elemdofb} vanish. 

For simplicity of presentation, the discussion below is focused on the case $k=2$.  Let $v:=w+c_1\lambda_1^2\lambda_2+c_2\lambda_2^2\lambda_3+c_3\lambda_3^2\lambda_1$ with parameters $c_1$, $c_2$, $c_3$, and $w$ could be any function in $P_2(T)$. The vanishing of \eqref{P2elemdofa} results in
\be
\label{12uni:1}
\mathscr{F}_{i,j}(v):=\frac{1}{\vert{F_i}\vert}\int_{F_i} v \lambda_j d{s}=0, \quad i, j\in \{1, 2, 3, 4\}, i\neq j.
\ee
Let $\mathscr{F}_1,\mathscr{F}_2,\mathscr{F}_3$ be the following linear functionals 
\begin{equation*}
\begin{aligned}
\mathscr{F}_1: &= \mathscr{F}_{1,2}-\mathscr{F}_{2,1} + \mathscr{F}_{3,1} - \mathscr{F}_{1,3} + \mathscr{F}_{2,3} - \mathscr{F}_{3,2},\\
\mathscr{F}_2: &= \mathscr{F}_{1,2} - \mathscr{F}_{2,1} + \mathscr{F}_{4,1} - \mathscr{F}_{1,4} + \mathscr{F}_{2,4} - \mathscr{F}_{4,2},\\
\mathscr{F}_3: &= \mathscr{F}_{1,3} - \mathscr{F}_{3,1} + \mathscr{F}_{4,1} - \mathscr{F}_{1,4}
+ \mathscr{F}_{3,4} - \mathscr{F}_{4,3}.
\end{aligned}
\end{equation*}
It can be verified that
$\mathscr{F}_i(w)=0$ for $i=1,2,3$.
This and \eqref{12uni:1} yield 
\begin{equation*}
0=\mathscr{F}_i(v)=\mathscr{F}_i(\lambda_1^2\lambda_2)c_1+\mathscr{F}_i(\lambda_2^2\lambda_3)c_2+\mathscr{F}_i(\lambda_3^2\lambda_1)c_3,\quad i = 1, 2, 3,
\end{equation*}
which implies $c_1+c_2+c_3=0,~c_1=0$, and $c_3=0.$
Thus $c_1=c_2=c_3=0$. This shows $v \in P_2(T)$. Therefore, the vanishing of \eqref{P2elemdofa} leads to $v = c\mathfrak{b}_{2} $ with parameter $c$.
For $k=3$, similar arguments will lead to $v = c^\prime\mathfrak{b}_{3}$ with parameter $c^\prime$. 
As a result, \eqref{P2elemdofb} concludes the proof for $k=2, 3$.
\end{proof}
The global space for the reduced $k$-order nonconforming finite element is defined by
\begin{align}
V_{h,k+1}^-:= \{v\in L^2(\Omega)|\ v|_T\in P_{k+1}^-(T), \text{for all}~T\in\mathcal{T}_h, v ~\text{is $M_{k-1}$-continuous}\}.\label{H3n}
\end{align}
\section{The finite element method for the Stokes problem}
The four newly proposed nonconforming finite elements are exploited to solve the Stokes problem in this section.

For $k=2, 3$, $V_{h,k+1}$ defined in \eqref{H2n} is used to approximate the velocity. To this end, define 
\begin{equation}
  \bm{V}_{h,k+1}: = V_{h,k+1}\times V_{h,k+1}\times V_{h,k+1}.
\end{equation}
Then the nonconforming approximation of the velocity field with the homogeneous boundary conditions in the weak sense is given by 
\begin{align}
    \bm{V}_{h,k+1,D}:=\{\bm{v}\in \bm{V}_{h,k+1}|\, \bm{v}~~ \text{is $M_{k-1}$-zero on}~~ \Gamma_D\}.
\end{align}
The associated pressure will be sought in the space
\begin{align}
    Q_{h,k}:=\{q\in L^2(\Omega)|\,\,  q|_T\in P_{k}(T)~~\text{for all}~~T\in \mathcal{T}_h\}.
\end{align}
It follows from the definition of $\bm{V}_{h,k+1,D}$ and $Q_{h,k}$ that $\operatorname{div}_h\bm{V}_{h,k+1,D}\subseteq Q_{h,k}$. Thus, the discrete velocity field $\bm{v}_h\in \bm{V}_{h,D,k+1}$ is element-wise divergence-free.

Next, for $k=2, 3$, $V_{h,k+1}^-$ given in \eqref{H3n} is used to approximate the velocity. Define  
\begin{equation}
  \bm{V}_{h,k+1}^-: = V_{h,k+1}^-\times V_{h,k+1}^-\times  V_{h,k+1}^-.
\end{equation}
And the nonconforming approximation of the velocity field with the homogeneous boundary conditions in the weak sense is given by 
\begin{align}
    \bm{V}_{h,k+1,D}^-:=\{\bm{v}\in \bm{V}_{h,k+1}^-|\, \bm{v}~~ \text{is $M_{k-1}$-zero on}~~ \Gamma_D\}.
\end{align}
The associated pressure will be sought in $Q_{h,k-1}$.

Let $(\bm{V}_{h,D}, Q_h)$ be either of $(\bm{V}_{h,k+1,D}, Q_{h,k})$ and $(\bm{V}_{h,k+1,D}^-, Q_{h,k-1}$). 
The nonconforming finite element method for \eqref{eq1} is to find $(\bm{u}_h,p_h)\in \bm{V}_{h,D}\times Q_h$, such that
\be
\label{eqb}
\left\{
\begin{aligned}
a_h(\bm{u}_h,\bm{v}_h)+b_h(\bm{v}_h,p_h)&=(\bm{f},\bm{v}_h)+(\bm{g},\bm{v}_h)_{\Gamma_N} &\quad \text{ for all } \bm{v}_h\in \bm{V}_{h,D},\\
b_h(\bm{u}_h,q_h)&=0 &\quad \text{ for all } q_h\in Q_h,
\end{aligned}
\right.
\ee
where 
\begin{align*}
a_h(\bm{u}_h,\bm{v}_h):=2\mu\sum\limits_{T\in\mathcal{T}_h}\int_T\varepsilon(\bm{u}_h):\varepsilon(\bm{v}_h)\, d{x}, \quad b_h(\bm{v}_h,q_h):=-\sum\limits_{T\in\mathcal{T}_h}\int_T \operatorname{div}\bm{v}_h\, q_h\, d{x}.
\end{align*}

Besides, for $\mathbb{R}^d$-valued function $\bm{v}$ with $d\in \{1, 2, 3\}$, define 
\begin{align*}
\Vert{\bm{v}}\Vert^2_{m,h}:=\sum\limits_{T\in\mathcal{T}_h}\Vert{\bm{v}}\Vert^2_{m,T}, \quad\vert{\bm{v}}\vert_{m,h}^2:=\sum\limits_{T\in\mathcal{T}_h}\vert{\bm{v}}\vert^2_{m,T},
\end{align*}
where $m$ is a non-negative integer.

\begin{remark}
The conditions $(\operatorname{H})$ and $(\operatorname{H}^\prime)$ imply that $|\cdot|_{1,h}$ is a norm of the nonconforming finite element space $\bm{V}_{h,D}$, see for example, \cite[Lemma 2]{m1973p}.
\end{remark}

\section{The discrete inf-sup condition}
The discrete inf-sup condition for the nonconforming finite element pairs is established in this section.
For $(\bm{V}_{h,k+1,D},Q_{h,k})$, the Fortin operator \cite[Proposition 2.8]{brezzi2012mixed} can not be constructed. The key to circumvent this is to analyze the divergence space of the bubble function space. For the reduced finite element pairs, the discrete inf-sup condition for $(\bm{V}^-_{h,3,D},Q_{h,1})$ is established by constructing the Fortin operator, while this is unfeasible for $(\bm{V}^-_{h,4,D},Q_{h,2})$. To deal with this problem, the macro-element technique \cite{stenberg1984analysis} is used.  

\subsection{The discrete inf-sup condition for $(\emph{\textbf{V}}_{h,k+1,D},Q_{h,k})$ }
To prove the discrete inf-sup condition for $(\bm{V}_{h,k+1,D},Q_{h,k})$, some crucial results are presented first. 

Let $P_{0,k}^\perp(T)$ be the orthogonal complement of $P_{0}(T)$ with respect to $P_{k}(T)$. Define $\bm{\mathfrak{B}}_{k+1}(T) := \mathfrak{B}_{k+1}(T)\times \mathfrak{B}_{k+1}(T)\times \mathfrak{B}_{k+1}(T)$. 
A key result is presented in the following lemma, but it runs into the problem of being hard to prove, since the dimension of $\operatorname{div}\bm{\mathfrak{B}}_{k+1}(T)$ is difficult to figure out. To overcome that, a constructive proof is adopted below.

\begin{lemma}\label{forBB1:1}
For $k=2,3$, there holds 
$\operatorname{div}\bm{\mathfrak{B}}_{k+1}(T)=P_{0,k}^\perp(T)$.
\end{lemma}
\begin{proof}
Given any function $\bm{v}\in \bm{\mathfrak{B}}_{k+1}(T)$, the definition of $\bm{\mathfrak{B}}_{k+1}(T)$ shows $\operatorname{div}\bm{v}\in P_k(T)$. In addition, an integration by parts leads to 
\begin{align*}
\int_T \operatorname{div}\bm{v}\, q ~dx=\sum_{F\in\mathcal{F}(T)}\int_F \bm{v}\cdot\mathbf{n}\, q ~ds=0~~\text{for all}~~q\in P_0(T).
\end{align*}
Thus $\operatorname{div}\bm{\mathfrak{B}}_{k+1}(T)\subseteq  P_{0,k}^\perp(T)$. It suffices to prove $P_{0,k}^\perp(T)\subseteq\operatorname{div}\bm{\mathfrak{B}}_{k+1}(T)$. 

For $k=2$, recall the eight linearly independent basis functions 
\begin{align*}
  &\phi_i=5\lambda_i^3-5\lambda_i^2+\lambda_i,~~ 1\leq i\leq 4, \\
  &\varphi_i = -2\lambda_i^2+4\lambda_i-2 + 3\sum_{j=1, j\neq i}^4\lambda_j^2+30\prod_{j=1,j\neq i}^4\lambda_j, ~~1\leq i\leq 4,
\end{align*}
in ${\mathfrak{B}}_{3}(T)$. Since $\phi_i\bm{\zeta}\in \bm{\mathfrak{B}}_{3}(T)$ for any vector $\bm{\zeta}\in \mathbb{R}^3$, there holds
\begin{align*}
    \operatorname{div}(\phi_i\bm{\zeta})=(15\lambda_i^2-10\lambda_i+1)(\nabla\lambda_i\cdot \bm{\zeta})\in \operatorname{div}\bm{\mathfrak{B}}_{3}(T).
\end{align*}
This shows 
\begin{equation}\label{divB3_1}
    15\lambda_i^2-10\lambda_i+1\in \operatorname{div}\bm{\mathfrak{B}}_3(T).
\end{equation}

Furthermore, let $F_i$ be the face of $T$ opposite to the vertex $a_i$ and $\bm{\nu}_{F_{i}}$ be the outward normal vector of $F_i$ with magnitude $2|F_i|$, $1\leq i\leq 4$. Notice that 
\be\label{lambda:1}
\nabla\lambda_i=-\frac{1}{3!|T|}\bm{\nu}_{F_i},
\ee
and the tangent vector $\bm{t}_{ij}=a_i-a_j$ is orthogonal to the face normal vector $\bm{\nu}_{F_k}$, where $i$, $j$, and $k$ are distinct integers in $\{1, 2, 3, 4\}$. Since $\varphi_i\bm{t}_{ij}\in\bm{\mathfrak{B}}_3(T)$, it follows from the chain rule and \eqref{lambda:1} that 
\begin{align*}
    \operatorname{div}(\varphi_i\bm{t}_{ij}) = (30\lambda_{k}\lambda_l +6 \lambda_j +4\lambda_i-4)(\nabla\lambda_j\cdot\bm{t}_{ij})\in \operatorname{div}\bm{\mathfrak{B}}_{3}(T),
\end{align*}
here $i$, $j$, $k$, and $l$ are distinct integers in $\{1, 2, 3, 4\}$.
Thus 
\be
\label{divB3eq1}
15\lambda_k\lambda_l+3\lambda_j+2\lambda_i-2\in \operatorname{div}\bm{\mathfrak{B}}_{3}(T), 
\ee
and by the symmetry of indices $i$ and $j$, there holds
\be
\label{divB3eq2}
15\lambda_k\lambda_l+3\lambda_i+2\lambda_j-2\in \operatorname{div}\bm{\mathfrak{B}}_{3}(T).
\ee
A subtraction of \eqref{divB3eq2} from \eqref{divB3eq1} yields $\lambda_i-\lambda_j\in \operatorname{div} \bm{\mathfrak{B}}_3(T)$. Therefore
\be\label{tools:1}
4\lambda_i-1=\sum_{j=1,i\neq j}^4(\lambda_i-\lambda_j)\in\operatorname{div}\bm{\mathfrak{B}}_{3}(T).
\ee
This and \eqref{divB3_1} imply $(10\lambda_i^2-1)\in \operatorname{div}\bm{\mathfrak{B}}_{3}(T)$ for $1\leq i\leq 4$. Besides, \eqref{tools:1} and \eqref{divB3eq1} show $(20\lambda_j\lambda_l-1)\in\operatorname{div}\bm{\mathfrak{B}}_{3}(T)$ with distinct $j$ and $l$ in $\{1, 2, 3, 4\}$. Thus
\begin{align*}
P_{0,2}^\perp(T)=\operatorname{span}\{10\lambda_i^2-1, 20\lambda_j\lambda_l-1|\,i, j, l\in\{1, 2, 3, 4\}, j\neq l\}\subseteq \operatorname{div}\bm{\mathfrak{B}}_{3}(T)
\end{align*} 
follows. For $k=3$, by virtue of the explicit representation of the bubble function \eqref{B4bb1}--\eqref{B4bb2}, $P_{0,3}^\perp(T)\subseteq \operatorname{div}\bm{\mathfrak{B}}_{4}(T)$ can be proved similarly, and the details are omitted here. Hence $P_{0,k}^\perp(T)\subseteq \operatorname{div}\bm{\mathfrak{B}}_{k+1}(T)$ holds for $k=2,3$. This concludes the proof. 
\end{proof}

In light of Lemma \ref{forBB1:1}, the following lemma holds.
\begin{lemma}\label{BB1:lemma:1}
For any $q_h\in Q_{h,k}$, if 
\be\label{p0:orth}
\int_T q_h\, dx = 0~~\text{for all}~~T\in \mathcal{T}_h, 
\ee
then there exists a $\bm{v}_h\in \bm{V}_{h,k+1,D}$ such that 
\begin{align}
    \operatorname{div}\bm{v}_h=q_h,\quad \text{and}\quad \|\bm{v}_h\|_{1,h}\lesssim \|q_h\|_0.\label{forBB1:2}
\end{align}
\end{lemma}
\begin{proof}
Given any $q_h\in Q_{h,k}$ satisfying \eqref{p0:orth}, 
Lemma \ref{forBB1:1}
implies there exists a $\bm{v}_T\in \bm{\mathfrak{B}}_{k+1}(T)$ such that $\operatorname{div}\bm{v}_T=q_h|_T$. Let $\bm{v}_h|_T=\bm{v}_T$ for all $T\in\mathcal{T}_h$, then $\bm{v}_h\in \bm{V}_{h,k+1,D}$ is obtained. Since the matching $\operatorname{div}\bm{v}_h=q_h$ is independently done on each tetrahedron $T\in\mathcal{T}_h$, by affine mapping and the scaling argument, \eqref{forBB1:2} holds. 
\end{proof}

\begin{lemma}\label{BB1:lemma:2}
Given any $q_h\in Q_{h,k}$, there exists a $\bm{v}_h\in \bm{V}_{h,k+1,D}$ such that
\be\label{bb1:step2:1}
\int_T (\operatorname{div}\bm{v}_h-q_h)\, dx=0~~\text{for all}~~T\in \mathcal{T}_h, \quad \text{and}\quad \|\bm{v}_h\|_{1,h}\lesssim \|q_h\|_0.
\ee
\end{lemma}
\begin{proof}
Given any $q_h\in Q_{h,k}$, it follows from the stability of the continuous formulation \cite[Corollary 2.4]{Girault1986FiniteEM}, there exists $\bm{v}\in \bm{H}_D^1(\Omega)$ such that 
\begin{equation*}
\operatorname{div}\bm{v}=q_h~~\text{and}~~\|\bm{v}\|_{1}\lesssim \|q_h\|_0.
\end{equation*}
For $k=2, 3$, and $\bm{v}\in\bm{H}_{D}^1(\Omega)$, define an interpolation $\Pi_{h,k+1}: \bm{H}_D^1(\Omega)\rightarrow \bm{V}_{h,k+1,D}$ by 
\begin{align}
    &\int_F \Pi_{h,k+1}\bm{v}\cdot \bm{q}\,ds=\int_F \bm{v}\cdot \bm{q}\,ds~~\text{for all}~~\bm{q}\in \bm{P}_{k-1}(F),\, F\in\mathcal{F}_h(\Omega)\cup \mathcal{F}_h(\Gamma_D),\label{pi1:1}\\
    &\int_T\Pi_{h,k+1}\bm{v}\cdot \bm{q}\, dx=\int_T\bm{v}\cdot \bm{q}\, dx~~\text{for all}~~\bm{q}\in \bm{\mathfrak{B}}_{k+1}(T),\, T\in\mathcal{T}_h.\label{pi1:2}
\end{align}
Due to the unisolvence of \eqref{facemomentP3}--\eqref{momentP3}, the interpolation $\Pi_{h,k+1}$ is well-defined, and 
\be\label{bounded:pi1}
\|\Pi_{h,k+1}\bm{v}\|_{1,h}\lesssim \|\bm{v}\|_{1}\lesssim \|q_h\|_0.
\ee
Let $\bm{v}_h:=\Pi_{h,k+1}\bm{v}$. The integration by parts and \eqref{pi1:1}--\eqref{pi1:2} lead to
\begin{align*}
    \int_T \operatorname{div}\bm{v}_h-q_h\, dx &= \int_{\partial{T}}\bm{v}_h\cdot \mathbf{n}\, ds -\int_T q_h\, dx\\
    &=\int_{\partial T}\bm{v}\cdot\mathbf{n}\, ds-\int_T q_h\, dx=\int_T \operatorname{div}\bm{v}-q_h\, dx=0.
\end{align*}
Thus \eqref{bb1:step2:1} follows.

\end{proof}

The discrete inf-sup condition is established in the following theorem by Lemma \ref{BB1:lemma:1} and Lemma \ref{BB1:lemma:2}.
\begin{theorem}\label{BB}
For $k=2, 3$, there exists a positive constant $c$ independent of the mesh-size such that 
\be
\label{infcon}
\sup_{\tv\in \bm{V}_{h,k+1,D}}\frac{b_{h}(q,\tv)}{\|\tv\|_{1,h}}\geq c \|q\|_{0} \quad \text{ for all }~~ q\in Q_{h,k}/\mathbb{R},
\ee
where $Q_{h,k}/\mathbb{R}:=\{q\in Q_{h,k}:\int_{\Omega}q\, d{x}=0\}$.
\end{theorem}
\begin{proof}
For $k=2, 3$, given any $q_h\in Q_{h,k}$, Lemma \ref{BB1:lemma:2} shows there exists a $\bm{v}_1\in \bm{V}_{h,k+1,D}$ such that 
\be\label{bb1:proof:eq1}
\int_T\operatorname{div}\bm{v}_1 -q_h\, dx=0~~\text{and}~~\|\bm{v}_1\|_{1,h}\lesssim \|q_h\|_0.
\ee
Since $\operatorname{div}\bm{V}_{h,k+1,D}\subseteq Q_{h,k}$, it follows from Lemma \ref{BB1:lemma:1} that there exists a $\bm{v}_2\in \bm{V}_{h,k+1,D}$ such that
\be\label{bb1:proof:eq2}
\operatorname{div}\bm{v}_2=q_h-\operatorname{div}\bm{v}_1\quad \text{and}\quad \|\bm{v}_2\|_{1,h}\lesssim \|\operatorname{div}\bm{v}_1-q_h\|_0.
\ee
Let $\bm{v}:=\bm{v}_1+\bm{v}_2$. This implies that 
\begin{equation*}
\operatorname{div}\bm{v}=q_h\quad \text{and}\quad \|\bm{v}\|_{1,h}\lesssim \|q_h\|_0.
\end{equation*}
Thus the discrete inf-sup condition \eqref{infcon} follows.
\end{proof}
\subsection{The discrete inf-sup condition for $(\emph{\textbf{V}}_{h,k+1,D}^{\,-},Q_{h,k-1})$} The discrete inf-sup condition for the reduced element pair $(\bm{V}_{h, k+1,D}^-, Q_{h,k-1})$ is proved in this subsection. 

It should be noted that the Fortin operator \cite[Proposition 2.8]{brezzi2012mixed} can be constructed for the case $k=2$ by using DoFs \eqref{P2elemdofa}--\eqref{P2elemdofb}.
However, this is unfeasible for the case $k=3$, as $\nabla (q|_T)\notin \bm{P}_0(T)$ for a given  $q\in Q_{h,2}$. To establish the discrete inf-sup condition for $(\bm{V}_{h,4,D}^-, Q_{h,2})$, the so-called macro-element technique introduced in \cite{stenberg1984analysis} is adopted.

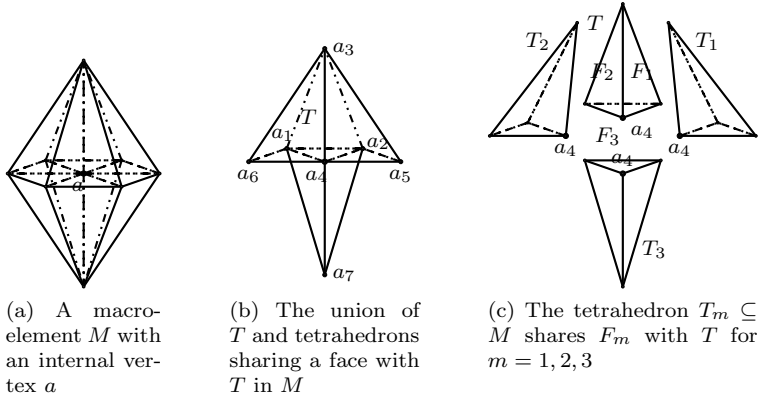
\begin{figure}[tbhp]
\captionsetup[subfigure]{font=footnotesize}
\centering
\subfloat[A macro-element $M$ with an internal vertex $a$]{\label{fig1:a}\begin{tikzpicture}[line width=0.3pt,scale=0.5]
  \coordinate (o) at (0,0);
 \coordinate (oa) at (0-0.15,0-0.0);
 
\coordinate (a1) at (2.0000, 0.0000);
\coordinate (a2) at (1.0000, 0.35);
\coordinate (a3) at (-1.0000, 0.35);
\coordinate (a4) at (-2.0000, 0.0000);
\coordinate (a5) at (-1.0000, -0.35);
\coordinate (a6) at (1.0000, -0.35);

  \coordinate (b) at (0,3);
  \coordinate (c) at (0,-3);
 \filldraw [black]

	     (o)        circle [radius=2pt]
	     
 	     (a1)      circle [radius=1pt]
              (a2)     circle [radius=1pt]
              (a3)     circle [radius=1pt]
              (a4)     circle [radius=1pt]
               (a5)     circle [radius=1pt]
               (a6)     circle [radius=1pt]

              (b)       circle [radius=1pt]
              (c)       circle [radius=1pt];;;;
              
   \draw[thick, dash dot dot] (o) -- (b) -- cycle;
   \draw[thick, dash dot dot] (o) -- (c) -- cycle;
             
   \draw[thick, dash dot dot] (o) -- (a1) -- cycle;
  \draw[thick, dash dot dot] (o) -- (a2) -- cycle;
  \draw[thick, dash dot dot] (o) -- (a3) -- cycle;
  \draw[thick, dash dot dot] (o) -- (a4) -- cycle;
  \draw[thick, dash dot dot] (o) -- (a5) -- cycle;
  \draw[thick, dash dot dot] (o) -- (a6) -- cycle;
  
   \draw[thick, dash dot dot] (a1) -- (a2) -- cycle;
    \draw[thick, dash dot dot] (a2) -- (a3) -- cycle;
     \draw[thick, dash dot dot] (a3) -- (a4) -- cycle;
      \draw[thick, fill=black!30] (a4) -- (a5) -- cycle;
      \draw[thick, fill=black!30] (a5) -- (a6) -- cycle;
       \draw[thick, fill=black!30] (a6) -- (a1) -- cycle;

  
  \draw[thick, fill=black!30] (a1) -- (b);
   \draw[thick, dash dot dot] (a2) -- (b); 
   \draw[thick, dash dot dot] (a3) -- (b); 
   \draw[thick, fill=black!30] (a4) -- (b); 
    \draw[thick, fill=black!30] (a5) -- (b); 
     \draw[thick, fill=black!30] (a6) -- (b);

    \draw[thick, fill=black!30] (a1) -- (c); 
   \draw[thick, dash dot dot] (a2) -- (c); 
   \draw[thick, dash dot dot] (a3) -- (c); 
   \draw[thick, fill=black!30] (a4) -- (c); 
   \draw[thick, fill=black!30] (a5) -- (c); 
   \draw[thick, fill=black!30] (a6) -- (c);

   \coordinate[label=below:$a$] (oa) at (oa);

\end{tikzpicture}}
\quad\quad\quad
\subfloat[
The union of $T$ and tetrahedrons sharing a face with $T$ in $M$
]{\label{fig1:b}\begin{tikzpicture}[line width=0.3pt,scale=0.5]
 \coordinate (o) at (0,0);
 \coordinate (t0) at (0,1.2);
  \coordinate (oa) at (0-0.2,0-0.02);
 
\coordinate (a1) at (2.0000, 0.0000);
\coordinate (a2) at (1.0000, 0.35);
\coordinate (a3) at (-1.0000, 0.35);
\coordinate (a4) at (-2.0000, 0.0000);

\coordinate (a3o) at (-1.1000, 0.34);
\coordinate (a2o) at (0.9000, 0.485);

  \coordinate (b) at (0,3);
  \coordinate (c) at (0,-3);
 \filldraw [black]

	          (o)        circle [radius= 2pt]
	     
 	         (a1)      circle [radius=1.5pt]
              (a2)     circle [radius=1.5pt]
              (a3)     circle [radius=1.5pt]
              (a4)     circle [radius=1.5pt]

              (b)       circle [radius=1.5pt]
              (c)       circle [radius=1.5pt];;;;
              
   \draw[thick, fill=black!30] (o) -- (b) -- cycle;
   \draw[thick, fill=black!30] (o) -- (c) -- cycle;
             
   \draw[thick,fill=black!30] (o) -- (a1) -- cycle;
  \draw[thick, dash dot dot] (o) -- (a2) -- cycle;
  \draw[thick, dash dot dot] (o) -- (a3) -- cycle;
  \draw[thick, fill=black!30] (o) -- (a4) -- cycle;

   \draw[thick, dash dot dot] (a1) -- (a2) -- cycle;
   \draw[thick, dash dot dot] (a2) -- (a3) -- cycle;
   \draw[thick, dash dot dot] (a3) -- (a4) -- cycle;

   \draw[thick, fill=black!30] (a1) -- (b);
   \draw[thick, dash dot dot] (a2) -- (b); 
   \draw[thick, dash dot dot] (a3) -- (b); 
   \draw[thick, fill=black!30] (a4) -- (b);

   \draw[thick, fill=black!30] (a2) -- (c); 
   \draw[thick, fill=black!30] (a3) -- (c);

   \coordinate[label=left:$T$] (t0) at (t0);

  \coordinate[label=above:$a_1$] (a3o) at (a3o);
  \coordinate[label=right:$a_2$] (a2o) at (a2o);
  \coordinate[label=right:$a_3$] (b) at (b);
  \coordinate[label=below:$a_4$] (oa) at (oa);
  \coordinate[label=right:$a_7$] (c) at (c);
  \coordinate[label=below:$a_6$] (a4) at (a4);
  \coordinate[label=below:$a_5$] (a1) at (a1);

\end{tikzpicture}}
\quad\quad\quad\subfloat[The tetrahedron $T_m\subseteq M$ shares $F_m$ with $T$ for $m=1,2,3$]{\label{fig1:c}\begin{tikzpicture}[line width=0.3pt,scale=0.5]
\coordinate (o1) at (1.5+0,0-0.5);
\coordinate (a11) at (1.5+2.0000, 0.0000-0.5);
\coordinate (a12) at (1.5+1.0000, 0.35-0.5);
\coordinate (b1) at (1.2,2.5);
\coordinate (b1o) at (1.7,2);

\coordinate (o2) at (0,-0.02);
\coordinate (a22) at (1.0000, 0.35);
\coordinate (a23) at (-1.0000, 0.35);
\coordinate (b2) at (0,3);   
\coordinate (b2o) at (-0.3,2.5);

\coordinate (c1o) at (-0.02,1.2);
\coordinate (c2o) at (0.00,1.2);
\coordinate (c3o) at (0.15,-0.5);

\coordinate (o3) at (-1.5+0,0-0.5);
\coordinate (a33) at (-1.5-1.0000, 0.35-0.5);
\coordinate (a34) at (-1.5-2.0000, 0.0000-0.5);
\coordinate (b3) at (-1.2,2.5);   
\coordinate (b3o) at (-1.7,2);

\coordinate (o4) at (0,0-1.5);
\coordinate (a42) at (1.0000, 0.35-1.5);
\coordinate (a43) at (-1.0000, 0.35-1.5);
\coordinate (b4) at (0,-4.5);        
 \coordinate (b4o) at (0.3,-3.5);

 \filldraw [black]

	      (o1)        circle [radius=2pt]
 	      (a11)      circle [radius=1pt]
              (a12)     circle [radius=1pt]
              (b1)       circle [radius=1pt]

               (o2)        circle [radius=2pt]
 	      (a22)      circle [radius=1pt]
              (a23)     circle [radius=1pt]
              (b2)       circle [radius=1pt]

              (o3)        circle [radius=2pt]
 	      (a33)      circle [radius=1pt]
              (a34)     circle [radius=1pt]
              (b3)       circle [radius=1pt]

	      (o4)        circle [radius=2pt]
 	      (a42)      circle [radius=1pt]
              (a43)     circle [radius=1pt]
              (b4)       circle [radius=1pt];;;;

            \draw[thick, fill=black!30] (o1) -- (a11) -- cycle;
           \draw[thick, dash dot dot] (o1) -- (a12) -- cycle;
           \draw[thick, dash dot dot] (a11) -- (a12) -- cycle;
      
          \draw[thick, fill=black!30] (o1) -- (b1) -- cycle;
          \draw[thick, fill=black!30] (a11) -- (b1) -- cycle;
	 \draw[thick, dash dot dot] (a12) -- (b1) -- cycle;

   	 \draw[thick, fill=black!30] (o2) -- (a22) -- cycle;
    	 \draw[thick, fill=black!30] (o2) -- (a23) -- cycle;
      	\draw[thick, dash dot dot] (a22) -- (a23) -- cycle;
      
      	 \draw[thick, fill=black!30] (o2) -- (b2) -- cycle;
       	 \draw[thick, fill=black!30] (a22) -- (b2) -- cycle;
	 \draw[thick, fill=black!30] (a23) -- (b2) -- cycle;

   	 \draw[thick, dash dot dot] (o3) -- (a33) -- cycle;
    	 \draw[thick, fill=black!30] (o3) -- (a34) -- cycle;
      	\draw[thick, dash dot dot] (a33) -- (a34) -- cycle;
      
      	 \draw[thick, fill=black!30] (o3) -- (b3) -- cycle;
       	 \draw[thick, dash dot dot] (a33) -- (b3) -- cycle;
	 \draw[thick, fill=black!30] (a34) -- (b3) -- cycle;

   	 \draw[thick,fill=black!30] (o4) -- (a42) -- cycle;
    	 \draw[thick, fill=black!30] (o4) -- (a43) -- cycle;
      	\draw[thick, fill=black!30] (a42) -- (a43) -- cycle;
      
      	 \draw[thick, fill=black!30] (o4) -- (b4) -- cycle;
       	 \draw[thick, fill=black!30] (a42) -- (b4) -- cycle;
	 \draw[thick, fill=black!30] (a43) -- (b4) -- cycle;

	  \coordinate[label=below:$a_4$] (o1) at (o1);   
	  \coordinate[label=below right:$a_4$] (o2) at (o2);
	   \coordinate[label=below:$a_4$] (o3) at (o3);
	    \coordinate[label=above:$a_4$] (o4) at (o4);

  \coordinate[label=left:$T$] (0t) at (b2o);   
  \coordinate[label=right:$T_1$] (1t) at (b1o);   
   \coordinate[label=left:$T_2$] (2t) at (b3o);  
   \coordinate[label=right:$T_3$] (3t) at (b4o);  
 
    \coordinate[label=right:$F _1$] (1c) at (c1o);  
    \coordinate[label=left:$F _2$] (2c) at (c2o);  
     \coordinate[label=left:$F _3$] (3c) at (c3o);

\end{tikzpicture}}
\caption{A macro-element $M$ and tetrahedron $T\subseteq M$ with $a_1$, $a_2$, $a_3$, $a_4$ as the four vertices}
\label{fig1}
\end{figure}
Suppose that each tetrahedron $T\in \mathcal{T}_h$ has at least one vertex in the interior of $\Omega$. Assume that $a\in \mathcal{V}_h(\Omega)$ is an internal vertex. Let the union of tetrahedrons with $a$ as one of the vertices be a macro-element $M$; see Figure \ref{fig1:a} for an illustration. For a macro-element $M$, let $\mathcal{F}_h(M)$ be the set interior facets in $M$ with respect to $\mathcal{T}_h$. 
For $(\bm{V}_{h,4,D}^-, Q_{h,2})$, define the corresponding finite element spaces on the macro-element $M$ by
\begin{align*}
   \bm{V}_{0,M}:=\{\tv\in \textbf{L}^{2}(M)~|~&\tv|_T\in \bm{P}_4^-(T)~\text{for all}~ T\subseteq {M}, \\
   &\bm{v}~\text{is $M_{2}$-continuous on}~ M ~\text{and}~\text{$M_2$-zero on}~\partial M\},
\end{align*}
and 
\[
Q_{0,M}:=\left\{q \in L^{2}(M)~|~q|_T\in P_2(T)\quad \text{for all}~~T\subseteq M,\, ~~ \int_{M}q\, d{x}=0 \right\}.
\]
Besides, define
\[
N_{M}:=\left\{q \in Q_{0,M} ~|~\int_{M}\operatorname{div}_h\tv \, q\, d{x}=0 \text{ for all } \tv \in \bm{V}_{0,M}\right\}.
\]
In light of the abstract theory from \cite{stenberg1984analysis}, to prove the inf-sup condition, it suffices to prove $N_{M}=\{0\}$. To this end, a key property for functions in $N_M$ is presented in the following lemma.

\begin{lemma}
Given any tetrahedron $T\subseteq M$, let $a_1$, $a_2$, $a_3$, $a_4$ be the vertices of $T$ and $a_{ij}=(a_i+a_j)/2$, $1\leq i< j\leq 4$. For any function $q\in N_M$, there holds
\be\label{paij:pakl}
q|_T(a_{ij}) = q|_T(a_{kl}).
\ee
Here $i,j,k$, and $l$ are distinct integers in $\{1, 2, 3, 4\}$ with $i<j$ and $k<l$.
\end{lemma}
\begin{proof}
For $q\in N_M$, $q|_T\in P_2(T)$ is of the form
\be
\label{p:expression}
q|_T = \sum_{i = 1}^4\lambda_i(2\lambda_i-1)q(a_i)+\sum_{1\leq j<k\leq 4}4\lambda_j\lambda_k q(a_{jk}).
\ee
Let $F_i$ be the face of $T$ opposite to $a_i$ and $\bm{\nu}_{F_i}$ be the outward normal vector of face $F_i$ with magnitude $\|\bm{\nu}_{F_i}\|=2|F_i|$, $i=1, 2, 3, 4$.
Notice that 
\be\label{lambda:1:1}
0=\nabla \left(\sum_{i=1}^4\lambda_i\right)=-\frac{1}{3!|T|}\sum_{i=1}^4{\bm{\nu}}_{F_i}.
\ee
This implies 
\be\label{lambda:2}
\bm{\nu}_{F_{1}}+\bm{\nu}_{F_{2}}=-(\bm{\nu}_{F_{3}}+\bm{\nu}_{F_{4}}).
\ee
Taking the gradient of \eqref{p:expression}, and then by \eqref{lambda:1:1}, it holds
\begin{align}
\nabla (q|_T) = \frac{-1}{3!|T|}&\left(\sum_{i = 1}^4(4\lambda_i-1)q(a_i)\bm{\nu}_{F_i}+\sum_{1\leq j<k\leq 4}4(\lambda_{k}\bm{\nu}_{F_{j}}+\lambda_j\bm{\nu}_{F_k}) q(a_{jk})\right).\label{grad:p:lemma}
\end{align}
Define $\bm{v}_h\in V_{0,M}$ by $\bm{v}_h:=\mathfrak{b}_{3}{\bm{\zeta}}$
with some vector ${\bm{\zeta}}\in \mathbb{R}^3$. Notice the support of $\bm{v}_h$ is $T$. Thus, an integration by parts yields
\begin{equation*}
\begin{split}
0 = \int_{M}\operatorname{div}_h{\bm{v}}_h \, q\, d{x}=\int_{T}\operatorname{div}{\bm{v}}_h \, q\, d{x}&=\sum_{i=1}^4\int_{F_i}\left\{\!\!\!\left\{\bm{v}_h\cdot \frac{\bm{\nu}_{F_i}}{2|F_i|}\right\}\!\!\!\right\}[\![q]\!]\, d{s}-\int_{T} \bm{v}_h\cdot \nabla q\, d{x}\\
&=-\int_{T} \mathfrak{b}_{3}\bm{\zeta}\cdot \nabla q\, d{x}.
\end{split}
\end{equation*}
This and \eqref{grad:p:lemma} result in  
\be\label{equ:Bgradientp}
\begin{split}
    0&= \sum_{i = 1}^4\frac{({\bm{\zeta}}\cdot\bm{\nu}_{F_i})}{3!|T|} q(a_i)\int_T \mathfrak{b}_{3}(4\lambda_i-1)\, d{x}\\
    &+\frac{4}{3!|T|}\sum_{1\leq j<k\leq 4}q(a_{jk})\left(({\bm{\zeta}}\cdot\bm{\nu}_{F_{j}})\int_{T}\mathfrak{b}_{3}\lambda_{k}\, d{x}+({\bm{\zeta}}\cdot\bm{\nu}_{F_{k}})\int_T \mathfrak{b}_{3}\lambda_j\, d{x}\right).
\end{split}
\ee
Elementary calculations show $\int_T \mathfrak{b}_{3}(4\lambda_i-1)\, d{x}=0$ for $i=1, 2, 3, 4$. This and the definition of $\mathfrak{b}_3$ lead to a simplification of \eqref{equ:Bgradientp}:
\begin{equation*}
0=\sum_{1\leq j<k\leq 4}{\bm{\zeta}}\cdot q(a_{jk})\left(\bm{\nu}_{F_{j}}+\bm{\nu}_{F_{k}}\right).
\end{equation*}
Due to the arbitrariness of $\bm{\zeta}\in \mathbb{R}^3$ and \eqref{lambda:2}, there holds 
\begin{equation*}
\begin{split}
(q(a_{34})-q(a_{12}))(\bm{\nu}_{F_{3}}+\bm{\nu}_{F_{4}})
&+(q(a_{24})-q(a_{13}))(\bm{\nu}_{F_{2}}+\bm{\nu}_{F_{4}})\\
&+(q(a_{14})-q(a_{23}))(\bm{\nu}_{F_{1}}+\bm{\nu}_{F_{4}})=0.
\end{split}
\end{equation*}
The vectors $\bm{\nu}_{F_{3}}+\bm{\nu}_{F_{4}}$, $\bm{\nu}_{F_{2}}+\bm{\nu}_{F_{4}}$, and $\bm{\nu}_{F_{1}}+\bm{\nu}_{F_{4}}$ are linearly independent. Thus \eqref{paij:pakl} follows.
\end{proof}

The following result for $N_M$ is proved by virtue of the above property. Indeed, the discrete inf-sup condition can be derived from the following lemma immediately.
\begin{lemma}\label{Nmis0}
It holds that $N_M =\{0\}$.
\end{lemma}
\begin{proof}
Given a tetrahedron $T\subseteq M$, let $a_1,a_2,a_3$, and $a_4$ be the four vertices of $T$. For $s= 1, 2, 3, 4$, let $F_s$ be the face of $T$ opposite to $a_s$, and $\bm{\nu}_{F_s}$ be its outward normal vector with magnitude $\|\bm{\nu}_{F_s}\|=2|F_s|$. Among these four facets, three of them are interior facets which are assumed to be $F_m$, $m=1, 2, 3$. Let $T_m\subseteq M$ be the tetrahedron which shares $F_m$ with $T$, and $a_{m+4}$ be the vertex opposite to $F_m$ in $T_m$. See Figure \ref{fig1:b}--\ref{fig1:c} for an illustration. 

Let $q\in N_{M}$, the formulae of $q$ on $T$ is
\be\label{formofp}
q|_{T}: = \sum_{i = 1}^4 c_i \lambda_{i}(2\lambda_{i}-1)+\sum_{1\leq j<k\leq 4}4c_{jk}\lambda_j\lambda_k,
\ee
with associated parameters $c_i$, $c_{jk}$ and the $i$-th barycentric coordinate $\lambda_{i}$ of $T$. The property \eqref{paij:pakl} implies $c_{12}= c_{34}, c_{13}=c_{24}, c_{23}=c_{14}$.
Therefore, it suffices to prove $c_1$, $c_2$, $c_3$, $c_4$, $c_{12}$, $c_{13}$, and $c_{23}$ are zero. 

To this end, let
$a^{m}_1$, $a^{m}_2$, $a^{m}_3$, $a^{m}_4$ be the four vertices of the tetrahedron $T_{m}$ for $m=1, 2, 3$, where 
\be\label{vertexSort}
\begin{aligned}
a^{1}_{1} =a_5, ~~a^{1}_{2} = a_2,~~ a^{1}_{3} = a_{4},~~ a^{1}_{4} = a_{3},\\
a^{2}_{1} =a_4, ~~a^{2}_{2} = a_6,~~ a^{2}_{3} = a_{3},~~ a^{2}_{4} = a_{1},\\
a^{3}_{1} =a_2, ~~a^{3}_{2} = a_1,~~ a^{3}_{3} = a_{7},~~ a^{3}_{4} = a_{4}.
\end{aligned}
\ee
Let $b^m_k = a_{(m+k-1)\operatorname{mod}4+1}$ be the $k$-th vertex of $F_{m}$ for $k=1, 2, 3$.
For $F_m=T\cap T_m$, let $\phi_{m,i,j}\in \bm{V}_{0,M}$ be the basis function dual to the degrees of freedom \eqref{P2elemdofa} on $F_m$, here $i,j\in\{1,2,3\}$ and $i\leq j$. Thus, the support of $\phi_{m,i,j}$ is $T\cup T_m$. For $T'$ being $T$ and $T_m$, $\phi_{m,i,j}|_{T'}\in P_4^-(T')$ satisfy 
\begin{align}
   &\frac{1}{\vert{F_m}\vert}\int_{F_m}\phi_{m,i,j}|_{T'}\, \lambda_{F_m,k}\lambda_{F_m,l}\, d{s}=\delta_{ik}\delta_{jl},\quad k, l\in \{1, 2, 3\},\label{facebasis:bb} \\ 
   &\frac{1}{\vert{F}\vert}\int_{F}\phi_{m,i,j}|_{T'}\,q\, d{s}=0, ~\text{for all}\, q\in P_2(F),~ F\in\mathcal{F}(T')\setminus\{F_m\}, \\
   &\frac{1}{|T'|}\int_{T'}\phi_{m,i,j}|_{T'}\, d{x}=0,\label{zeromoments}
\end{align}
where $\lambda_{F_m, k}$ is the $k$-th barycentric coordinate of $F_m$. In reality, $\lambda_{F_m, k}$ is equivalent to $\lambda_{(m+k-1) \operatorname{mod}4+1}$ of $T$ on $F_m$.

For $q\in N_M$, an integration by parts leads to 
\begin{align}
   0=\int_{M}q \operatorname{div}_h\bm{v}_h\, d{x}=\sum_{F\in \mathcal{F}_h(M)}\int_{F}\left\{\!\!\!\left\{{\bm{v}}_h\cdot\frac{\bm{\nu}_{F}}{2|F|}\right\}\!\!\!\right\}[\![q]\!]\, d{s}-\sum_{T'\subset M}\int_{T'}{\bm{v}}_h\cdot\nabla q\, d{x} \label{pM}
\end{align}
for all $\bm{v}_h\in \bm{V}_{0,M}$. In particular, some vector valued functions in $\bm{V}_{0,M}$ are defined to prove $q$ is zero on the macro-element $M$.

Firstly, take
${\bm{v}}_h:=(2\phi_{1,1,2}+\phi_{1,1,3})\bm{\nu}_{F_1}^\perp$. Since 
$F_1=T\cap T_1$, $(2\phi_{1,1,2}+\phi_{1,1,3})$ vanishes except on $T\cup T_1$. Besides, it can be verified by  calculations that all of the average and the first order moments of $(2\phi_{1,1,2}+\phi_{1,1,3})$ on $T_1$ vanish, while 
\begin{align}
&\frac{1}{|T|}\int_{T}(2\phi_{1,1,2}+\phi_{1,1,3})\lambda_1\,d{x} = -\frac{27}{4}, ~&&\frac{1}{|T|}\int_{T}(2\phi_{1,1,2}+\phi_{1,1,3})\lambda_2 \,d{x}= \frac{9}{4},\label{firstmoments}\\
&\frac{1}{|T|}\int_{T}(2\phi_{1,1,2}+\phi_{1,1,3})\lambda_3\,d{x} = \frac{9}{4} ,~&&\frac{1}{|T|}\int_{T}(2\phi_{1,1,2}+\phi_{1,1,3})\lambda_4\,d{x} = \frac{9}{4}.\label{firstmoments1}
\end{align}
Thus there holds 
\be
\label{eq:bb:1}
\begin{split}
0&=\sum_{F\in \mathcal{F}_h(M)}\int_{F}\left\{\!\!\!\left\{{\bm{v}}_h\cdot\frac{\bm{\nu}_{F}}{2|F|}\right\}\!\!\!\right\}[\![q]\!]\, d{s}-\sum_{T'\subseteq M}\int_{T'}{\bm{v}}_h\cdot\nabla q\, d{x}\\
&= \int_{F_1}\left\{\!\!\!\left\{{\bm{v}}_h\cdot\frac{\bm{\nu}_{F_1}}{2|F_1|}\right\}\!\!\!\right\}[\![q]\!]\, d{s}-\int_{T}{\bm{v}}_h\cdot\nabla q\, d{x}\\
&=-\int_{T}((2\phi_{1,1,2}+\phi_{1,1,3})\bm{\nu}_{F_1}^\perp)\cdot\nabla q\, d{x}.
\end{split}
\ee
Taking the gradient of \eqref{formofp}, and recall 
$\nabla \lambda_i=-\frac{1}{3!|T|}{\bm{\nu}}_{F_i}$
introduced in \eqref{lambda:1}. The formulae of $\nabla (q|_T)$ is 
\begin{align}
    \nabla (q|_{T}) = \frac{-1}{3!|T|}\left(\sum_{i = 1}^4(4\lambda_i-1)c_i\bm{\nu}_{F_i}+\sum_{1\leq j<k\leq 4}4(\lambda_{k}\bm{\nu}_{F_{j}}+\lambda_j\bm{\nu}_{F_k}) c_{jk}\right).\label{gradientpont0}
\end{align}
This and \eqref{eq:bb:1} lead to 
\begin{align*}
    0= &\frac{\bm{\nu}_{F_1}^\perp}{3!|T|}\cdot\left(\sum_{i = 1}^4 c_i\, \bm{\nu}_{F_i} \int_{T}(2\phi_{1,1,2}+\phi_{1,1,3})(4\lambda_i-1)\, d{x}\right.\\
    &\qquad\left.+4\sum_{1\leq j<k\leq 4}\int_{T}(2\phi_{1,1,2}+\phi_{1,1,3})(\lambda_{k}\bm{\nu}_{F_{j}}+\lambda_j\bm{\nu}_{F_k}) c_{jk}\, d{x}\right).
\end{align*}
Then \eqref{zeromoments} and \eqref{firstmoments}--\eqref{firstmoments1} yield
\begin{equation*} 
\begin{split}
0=\bm{\nu}_{F_1}^\perp\cdot &\left((-3c_1+c_{12}+c_{13}+c_{23})\bm{\nu}_{F_1}+(c_2-3c_{12}+c_{13}+c_{23})\bm{\nu}_{F_2}\right.\\
&+\left.(c_3+c_{12}-3c_{13}+c_{23})\bm{\nu}_{F_3}+(c_4+c_{12}+c_{13}-3c_{23})\bm{\nu}_{F_4}\right).
\end{split}
\end{equation*}
This plus \eqref{lambda:2} gives rise to 
\begin{align}
   c_{2}-c_{3}-4c_{12}+4c_{13}=0,~~~ c_{2}-c_{4}-4c_{12}+4c_{23}=0. \label{leq:1}
\end{align}

Secondly, take $\bm{v}_h:= (\phi_{2,1,2}+4\phi_{2,2,2})\bm{\nu}_{F_2}^\perp$ in \eqref{pM}. Since $F_2=T\cap T_2$, the support of $\phi_{2, i, j}$ is $T\cup T_2$. Furthermore, it can be verified by elementary calculations that all of the average and the first order moments of $(\phi_{2,1,2}+4\phi_{2,2,2})$ on $T_2$ vanish. Thus, \eqref{pM} is simplified to 
\[0=\int_{T}((\phi_{2,1,2}+4\phi_{2,2,2})\bm{\nu}_{F_2 }^\perp)\cdot\nabla q\, d{x}.\]
This and the formulae of $\nabla q$ on $T$ from \eqref{gradientpont0} lead to 
\begin{equation}\label{secondly:1}
\begin{aligned}
0= \frac{\bm{\nu}_{F_2}^\perp}{3!|T|}\cdot(&\sum_{i = 1}^4  c_i \, \bm{\nu}_{F_i}\int_{T}(\phi_{2,1,2}+4\phi_{2,2,2})(4\lambda _i-1)\, d{x}\\
&+\sum_{1\leq j<k\leq 4}4\int_{T}(\phi_{2,1,2}+4\phi_{2,2,2})(\lambda _{k}\bm{\nu}_{F_{j}}+\lambda _j\bm{\nu}_{F_k}) c_{jk} \, d{x}).
\end{aligned}
\end{equation}
Since the value of 
$\frac{1}{|T|}\int_{T}(\phi_{2,1,2}+4\phi_{2,2,2})\lambda_i\,d{x}$ is ${11}/{4}$,\,  $-{1}/{4}$,\,  $-{5}/{4}$,\, and $-{5}/{4}$ for $i$ being 1, 2, 3, and 4, respectively. This, \eqref{zeromoments}, \eqref{lambda:2}, and \eqref{secondly:1} result in
\begin{align}
5(c_{4} -c_{3} )+12(c_{13} -c_{23})=0,\quad 11c_{1}+5c_{3} +4(c_{12} -c_{23})-16c_{13} =0.\label{leq:2} 
\end{align}

Thirdly, take $\bm{v}_h:= (6\phi _{3,1,1}+2\phi _{3,1,2}+\phi _{3,1,3})\bm{\nu}_{F_3 }^\perp$ in \eqref{pM}. Since $F_3 =T\cap T_3$, the support of $\phi _{3, i, j}$ is $T\cup T_3$. Furthermore, it can be verified that all of the average and the first order moments of $(6\phi _{3,1,1}+2\phi _{3,1,2}+\phi _{3,1,3})$ on $T_3$ vanish, while on $T$, the value of $\frac{1}{|T|}\int_{T}(6\phi _{3,1,1}+2\phi _{3,1,2}+\phi _{3,1,3})\lambda_i\,d{x}$ is $-{27}/{4}$,\,  $-{3}/{4}$,\,  ${9}/{4}$,\, and ${21}/{4}$ for $i$ being 1, 2, 3, and 4, respectively. 
By similar arguments as in the last two regimes where $\bm{v}_h$ are defined by the basis functions on $F _1$ and  $F _2$, respectively, it can be derived that
\begin{align}
9c_{1} -c_{2} -8c_{12}-4(c_{23} -c_{13})=0,~ c_{2} +7c_{4} -8c_{13} -12(c_{23} -c_{12} ) =0. \label{leq:3}
\end{align}

Finally, \eqref{leq:1}, \eqref{leq:2}, and \eqref{leq:3} bring $c_1=c_2=c_3=c_4=c_{12}=c_{13}=c_{23}$. Hence $q$ is a constant on $T$, and the symmetry of the macro-element $M$ shows that $q$ is a constant on any tetrahedron $T\subseteq M$. Thus, the equation \eqref{pM} turns into
\be\label{new:pM}
0=\int_{M}q \operatorname{div}_h\bm{v}_h\, d{x}=\sum_{F\in \mathcal{F}_h(M)}\int_{F}\left\{\!\!\!\left\{{\bm{v}}_h\cdot\frac{\bm{\nu}_{F}}{2|F|}\right\}\!\!\!\right\}[\![q]\!]\, d{s}.
\ee
For $F\in\mathcal{F}_h(M)$ with $\lambda_a$, $\lambda_b$, $\lambda_c$ as the three barycentric coordinates, define a face bubble function by
$B_{F,3}:=27\lambda_a\lambda_b\lambda_c$. Take ${\bm{v}}_h:=B_{F,3}\frac{\bm{\nu}_F}{2|F|}$ in \eqref{new:pM}, there holds
\begin{equation*}
\int_{F}B_{F,3}[\![q]\!]\, d{s}=[\![q]\!]=0.
\end{equation*}
This implies that $q$ is a constant on $M$, and $\int_M q\, d{x}=0$ shows the constant is zero. Consequently, $N_M=\{0\}$.
\end{proof}

\begin{theorem}\label{BB:plus}
For $k=2, 3$, there exists a positive constant $c$ independent of the mesh-size such that 
\be
\label{infcon:plus}
\sup_{\tv\in \bm{V}_{h,k+1,D}^-}\frac{b_{h}(q,\tv)}{\|\tv\|_{1,h}}\geq c \|q\|_{0} \quad \text{ for all }~~ q\in Q_{h,k-1}/\mathbb{R},
\ee
where $Q_{h,k-1}/\mathbb{R}:=\{q\in Q_{h,k-1}:\int_{\Omega}q\, d{x}=0\}$.
\end{theorem}
\begin{proof}
For $k=2$, the Fortin operator can be constructed, thus \eqref{infcon:plus} follows from \cite[Proposition 2.8]{brezzi2012mixed}. For $k=3$, Lemma \ref{Nmis0} shows $N_M=\{0\}$. As a result, \eqref{infcon} follows from \cite[Theorem 3.1]{stenberg1984analysis}.
\end{proof}

\section{The error analysis}
 The error estimates are presented in this section.

\begin{theorem}
For $k=2, 3$, let $(\bm{u},p)\in \bm{H}^{k+1}(\Omega)\times H^k(\Omega)$ be the solution of \eqref{eq1}. 
The problem \eqref{eqb} admits a unique solution $(\bm{u}_h,p_h)\in \bm{V}_{h,D}\times Q_h$, 
where $\bm{V}_{h,D}\times Q_h$ is either of $\bm{V}_{h,k+1,D}\times Q_{h,k}$ and $\bm{V}_{h,k+1,D}^-\times Q_{h,k-1}$. Furthermore, there holds
\begin{align}
 &\vert{\bm{u}-\bm{u}_h}\vert_{1,h}+\Vert{p-p_h}\Vert_{0}\lesssim h^k(\vert{\bm{u}}\vert_{k+1}+\vert{p}\vert_k),   \label{gj1}\\
 &\Vert{\bm{u}-\bm{u}_h}\Vert_{0,h}+\Vert{p-p_h}\Vert_{-1}\lesssim h^{k+1}(\vert{\bm{u}}\vert_{k+1}+\vert{p}\vert_k).\label{gj2}
\end{align}
Here $\|\cdot\|_{-1}$ denotes the norm of the dual space of $H^1(\Omega)\cap L_0^2(\Omega)$.
\end{theorem}

\begin{proof} 
The $M_{k-1}$-continuity of $\bm{V}_{h,D}$ implies the discrete Korn's inequality in light of the analysis in \cite{brenner2004korn}. This plus the discrete inf-sup condition \eqref{infcon} and \eqref{infcon:plus} shows the problem admits a unique solution.
Then the estimates \eqref{gj1}--\eqref{gj2} follow from the abstract error estimates of mixed nonconforming finite elements \cite[Proposition 5.5.6]{boffi2013mixed} and the consistency error estimates \cite[Lemma 3]{m1973p}.
\end{proof}

\section{Numerical experiments}
The numerical experiments are presented in this section to verify the error analysis and convergence results. 
 The computation domain is the cube $\Omega=(0,1)^3$ with Neumann boundary $\Gamma_N=\{z=1\}\cap \partial\Omega$ and Dirichlet boundary $\Gamma_D=\partial\Omega\setminus\Gamma_N$. The load function $\bm{f}=-2\mu \operatorname{div}\varepsilon(\bm{u})+\nabla p$ in \eqref{eq1} is derived by the following exact solutions:
\begin{align*}
    \bm{u}=\operatorname{curl}\begin{pmatrix}
    y^2(1-y)^2x(1-x)z^2(1-z)^3\\
    x^2(1-x)^2y(1-y)z^2(1-z)^3\\
    0
    \end{pmatrix},\quad p =(x-\frac{1}{2})(y-\frac{1}{2})(1-z).
\end{align*}
In this regime, it can be verified that $\bm{g}|_{\Gamma_N}=(\varepsilon(\bm{u})-p\bm{I})\bm{\nu}|_{\Gamma_N}=0$, and $\bm{u}|_{\Gamma_{D}}=\bm{u}_D=0$.


Both of $(\bm{V}_{h,k+1,D}, Q_{h,k})$ and $(\bm{V}_{h,k+1,D}^-, Q_{h,k-1})$ with $k=2, 3$ are exploited for \eqref{eqb}. The numerical results are performed in Figure \ref{fig:rate} with \cite{yfem} used in the computation.
\begin{figure}[h]
\captionsetup[subfigure]{font=footnotesize}
\centering
\subfloat[The convergence rate of $\|\bm{u}-\bm{u}_h\|_{1,h}$.]{\label{converfig:a} \includegraphics[width=0.5\textwidth]{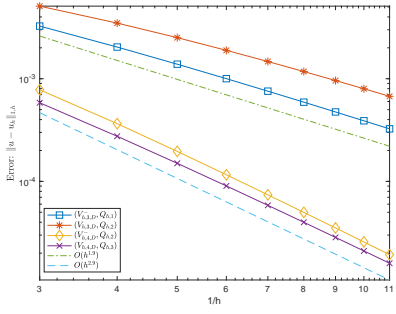}}
\subfloat[The convergence rate of $\|p-p_h\|_{0}$.]{\label{converfig:b}\includegraphics[width=0.5\textwidth]{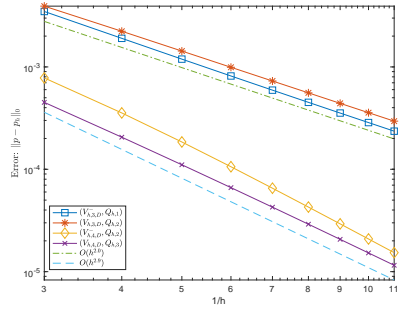}}\\
\subfloat[The convergence rate of $\|\bm{u}-\bm{u}_h\|_{0}$.]{\label{converfig:c}\includegraphics[width=0.5\textwidth]{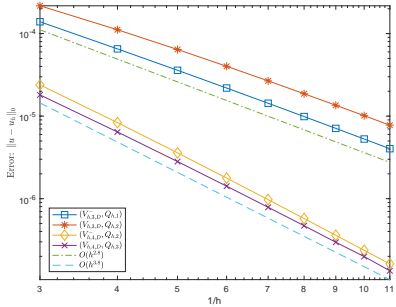}}
\subfloat[The number of DoFs.]{\label{converfig:d}\includegraphics[width=0.5\textwidth]{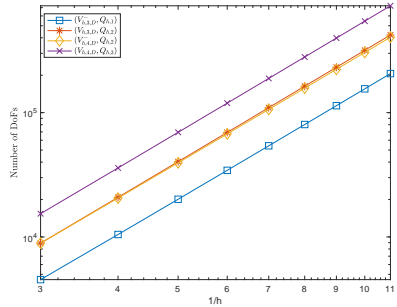}}
\caption{Convergence rates and the number of DoFs}
\label{fig:rate}
\end{figure}

As the estimates presented in \eqref{gj1}, it can be seen in Figure \ref{converfig:a}--\ref{converfig:b} that for the finite element pairs $(\bm{V}_{h,3,D}, Q_{h,2})$ and $(\bm{V}_{h,3,D}^-, Q_{h,1})$, the convergence rates of both the velocity in the energy norm and the pressure in the $L^2$ norm are close to 2, while these rates are 3 for the finite element pairs $(\bm{V}_{h,4,D}, Q_{h,3})$ and $(\bm{V}_{h,4,D}^-, Q_{h,2})$. In the interim, Figure \ref{converfig:c} shows that the convergence rates of the velocity in the $L^2$-norm for these methods are close to 3 and 4 when $k$ is 2 and 3, respectively. This is consistent with the the estimates \eqref{gj2}.

The number of total DoFs of these finite element pairs is presented in Figure \ref{converfig:d}. In contrast with $(\bm{V}_{h,k+1,D}, Q_{h,k})$, the reduced finite element spaces $(\bm{V}_{h,k+1,D}^-, Q_{h,k-1})$ have fewer DoFs while retaining the same convergence rates. 

%


\bibliographystyle{spmpsci}      
\bibliography{bibfile}


\end{document}